\documentclass{amsart}
\usepackage{amssymb,amsmath,latexsym}
\usepackage{times}

\numberwithin{equation}{section}

\newtheorem{thm}{Theorem}[section]

\newtheorem{lem}[thm]{Lemma}

\newtheorem{conj}[thm]{Conjecture}

\newcommand{\del}{\backslash}
\newcommand{\cl}{\hbox{\rm cl}}

\title[Intertwining connectivity in matroids]
{Intertwining connectivity in matroids}

\date{\today}

\author[Chen]{Rong Chen}
\address{Center for Discrete Mathematics, Fuzhou University,
Fuzhou, P. R. China} %+86 59183851461

\author[Whittle]{Geoff Whittle}
\address{School of Mathematics, Statistics and Operations Research,
Victoria University of Wellington, New Zealand}%+64 44635650

\thanks{This research was supported by a grant from the Marsden Fund of
New Zealand, and grants from China with number CNNSF (No.11201076), SRFDP (No.20113514120010), CSC and JA11032}
\begin{document}

\begin{abstract}
Let $M$ be a matroid and let $Q$, $R$, $S$ and $T$ be subsets of the ground
set such that the smallest separation that separates $Q$ from $R$ has order
$k$ and the smallest separation that separates $S$ from $T$ has order $l$.
We prove that if $E(M)-(Q\cup R\cup S\cup T)$ is sufficiently large,
then there is an element $e$ of $M$ such that, in one of $M\backslash e$
or $M/e$, both connectivities are preserved.
\end{abstract}


%\footnote{
%Email: rongchen@fzu.edu.cn.

%The research supported partially by CNNSF (No.11201076), SRFDP (No.20113514120010) and JA11032.  }

\maketitle

\section{Introduction}
Let $M$ be a matroid with ground set $E(M)$.  For any $X\subseteq E(M)$, define $\lambda_M(X):=r_M(X)+r_M(E(M)-X)-r(M)$. For disjoint subsets $Q,R$ of $E(M)$, the {\sl connectivity between $Q$ and $R$} is \[\kappa_M(Q,R):=\text{min}\{\lambda_M(X): Q\subseteq X\subseteq E(M)-R\}.\]

In the paper, we prove
\begin{thm}\label{main-thm}
There is a function $c: \Bbb{N}^2\rightarrow \Bbb{N}$ with the following property. Let $M$ be a matroid, and $Q,R,S,T,F\subseteq E(M)$ sets of elements such that $Q\cap R=S\cap T=\emptyset$ and $F=E(M)-(Q\cup R\cup S\cup T)$. Let $k:=\kappa_M(Q,R)$ and $\ell:=\kappa(S,T)$.  If $|F|\geq c(k,\ell)$, then there is an element $e\in F$ such that one of the following holds:
\begin{itemize}
    \item[(i)] $\kappa_{M\del e}(Q,R)=k$ and $\kappa_{M\del e}(S,T)=\ell$;
    \item[(ii)] $\kappa_{M/ e}(Q,R)=k$ and $\kappa_{M/ e}(S,T)=\ell$.
\end{itemize}
\end{thm}

This theorem resolves a conjecture of Geelen (private communication). It strengthens a
theorem of Huynh and van Zwam \cite{HZ} who prove the result for a class that includes
all representable matroids but does not include all matroids.

The value that we give for $c(k,\ell)$ is unlikely to be tight. The
$(k+1)\times (\ell+1)$ grid gives an example where the theorem fails with $|F|=2kl-l-k$.
Perhaps this example is extremal?

\begin{conj}
Theorem~\ref{main-thm} holds with $|F|=2kl-l-k+1$.
\end{conj}

\section{Proof of Theorem \ref{main-thm}}
For any disjoint subsets $Q,R$ of the ground set of a matroid $M$, Tutte \cite{Tutte65} proved that there is a minor $N$ of $M$ with $E(N)=Q\cup R$ and such that $\kappa(Q,R)=\lambda_N(Q)$, which is a generalization of Menger's theorem to matroids. Equivalently, we have
\begin{lem}\label{Tutte's th}
Let $M$ be a matroid and $Q,R$ be disjoint subsets of $E(M)$. For any $e\in E(M)-(Q\cup R)$ either $\kappa_{M\del e}(Q,R)=\kappa_M(Q,R)$ or $\kappa_{M/ e}(Q,R)=\kappa_M(Q,R)$.
\end{lem}

Let $M$ be a matroid and $Q,R$ be disjoint subsets of $E(M)$. Define $\sqcap_M(Q,R):=r_M(Q)+r_M(R)-r_M(Q\cup R)$. A partition $(A,B)$ of $E(M)$ is {\sl $Q-R$-separating of order $k+1$} if $Q\subseteq A$, $R\subseteq B$ and $\lambda_M(A)\leq k$. Let $e\in E(M)-(Q\cup R)$. If $\kappa_{M\del e}(Q,R)=\kappa_{M}(Q,R)$, then $e$ is {\sl deletable with respect to} $(Q,R)$; if $\kappa_{M/ e}(Q,R)=\kappa_{M}(Q,R)$, then $e$ is {\sl contractible with respect to} $(Q,R)$; and if $e$ is both deletable and contractible with respect to $(Q,R)$, then $e$ is {\sl flexible with respect to} $(Q,R)$. Lemma \ref{Tutte's th} implies that for any $e\in E(M)-(Q\cup R)$ either $e$ is deletable with respect to $(Q,R)$ or $e$ is contractible with $(Q,R)$.

\begin{thm}\label{non-flexible}(\cite{HZ}, Theorem 3.4.)
Let $M$ be a matroid and $Q,R$ be disjoint subsets of $E(M)$, let $k:=\kappa(Q,R)$, and let $F\subseteq E(M)-(Q\cup R)$ be a set of non-flexible elements. There are an ordering $(f_1,\cdots, f_n)$ of $F$ and a sequence of $(A_1,\cdots,A_n)$ of subsets of $E(M)$ such that
\begin{itemize}
    \item[(i)] $A_i$ is $Q-R$-separating of order $k+1$ for each $i\in\{1,\cdots,n\}$;
    \item[(ii)] $A_i\subseteq A_{i+1}$ for each $i\in\{1,\cdots,n\}$;
    \item[(iii)] $A_i\cap F=\{f_1,\cdots,f_i\}$  for each $i\in\{1,\cdots,n\}$;
    \item[(iv)] $f_i\in\cl(A_i-\{f_i\})\cap\cl(E(M)-A_i)$ or $f_i\in\cl^*(A_i-\{f_i\})\cap\cl^*(E(M)-A_i)$.
\end{itemize}
\end{thm}

\begin{thm}\label{keep-flexible}(\cite{HZ}, Lemma 3.6.)
Let $M$ be a matroid and $Q,R$ be disjoint subsets of $E(M)$, let $k:=\kappa(Q,R)$, and let $(U,E(M)-U)$ be a $Q-R$-separating set of order $k+1$. If $e\in E(M)-(U\cup R)$ is non-contradictable with respect to $(Q,R)$, then $e$ is also non-contradictable with respect to $(U,R)$.
\end{thm}

First we prove that Theorem \ref{main-thm} holds for the case $|S|=|T|=\ell$.

\begin{lem}\label{step-1}
There is a function $c: \Bbb{N}^2\rightarrow \Bbb{N}$ with the following property. Let $M$ be a matroid, and $Q,R,S,T,F\subseteq E(M)$ sets of elements such that $Q\cap R=S\cap T=\emptyset$ and $F=E(M)-(Q\cup R\cup S\cup T)$. Let $k:=\kappa_M(Q,R)$ and $\ell:=\kappa_M(S,T)$.  If $|S|=|T|=\ell$ and $|F|\geq c(k,\ell)$, then there is an element $e\in F$ such that one of the following holds:
\begin{itemize}
    \item[(i)] $\kappa_{M\del e}(Q,R)=k$ and $\kappa_{M\del e}(S,T)=\ell$;
    \item[(ii)] $\kappa_{M/ e}(Q,R)=k$ and $\kappa_{M/ e}(S,T)=\ell$.
\end{itemize}
\end{lem}

\begin{proof}
We prove that the result holds for $c(k,\ell):=(2\ell+1)2^{2k+1}$. If $F$ contains some flexible element with respect to $(Q,R)$ or $(S,T)$, then we are done. So we may assume that each element in $F$ is non-flexible with respect to  $(Q,R)$ and non-flexible with respect to $(S,T)$. By Lemma \ref{Tutte's th} an element $e$ in $F$ is deletable (or contractible) with respect to $(Q,R)$ if and only if $e$ is contractible (or deletable) with respect to $(S,T)$, for otherwise the lemma holds.

Let $(A_1, \cdots, A_{c(k,\ell)})$ be the nested sequence of $Q-R$ separating sets from Theorem \ref{non-flexible}, let $(B_1,\cdots, B_{c(k,\ell)})$ be their complements, and let $(f_1, \cdots, f_{c(k,\ell)})$ be the corresponding ordering of $F$. Since $|S|=|T|=\ell$, there is a positive integer $i$ such that $i+2^{2k+1}\leq c(k,\ell)$ and such that $Q\cup R\cup S\cup T\subseteq A_i\cup B_{i+2^{2k+1}}$. Set
 \[\begin{aligned}
&Q^{'}:=A_i,\  R^{'}:=B_{i+2^{2k+1}},\ F^{'}:=E(M)-(Q^{'}\cup R^{'}),\\
&A_j^{'}:=A_{i+j},\ B_j^{'}:=B_{i+j},\ f_j^{'}:=f_{i+j},\ \text{for\ any}\ 1\leq j\leq 2^{2k+1}.
\end{aligned}\]
That is, $F^{'}=\{f_1^{'},\cdots,f_{2^{2k+1}}^{'}\}$. By duality and Lemma \ref{keep-flexible}, each element in $F^{'}$ is non-flexible with respect to $(Q^{'},R^{'})$.

Let $(C_1, \cdots, C_{2^{2k+1}})$ be the nested sequence of $S-T$ separating sets from Theorem \ref{non-flexible} determined by the non-flexible-element set $F^{'}$ with respect to $(S,T)$, let $(D_1,\cdots, D_{2^{2k+1}})$ be their complements, and let $(g_1, \cdots, g_{2^{2k+1}})$ be the corresponding ordering of $F^{'}$.
By duality we may assume that $g_1$ is a deletable element with respect to $(S,T)$. Then (i) $g_1\in\cl(C_1-\{g_1\})$ and (ii) $g_1$ is a contractible element with respect to $(Q,R)$. By (i)
and the fact that $C_1-\{g_1\}\subseteq Q'\cup R'$ we see that $g_1\in\cl(Q'\cup R')$.
From (ii) we deduce that $g_1\notin\cl(Q^{'})$ and $g_1\notin\cl(R^{'})$. Therefore
$\sqcap_M(Q^{'}\cup\{g_1\},R^{'})=\sqcap_M(Q^{'},R^{'})+1$. Assume that $g_1=f_j^{'}$. If $j\leq 2^{2k}$ then set $Q^{''}:=A_j^{'}, R^{''}:=R^{'}$; else if $j>2^{2k}$ then set $Q^{''}:=Q^{'}, R^{''}:=B_{j-1}^{'}$. No matter which case happens, set $F^{''}:=E(M)-(Q^{''}\cup R^{''})$. Evidently, $|F^{''}|\geq 2^{2k}$ as $|F^{'}|=2^{2k+1}$. Replacing $Q^{'},R^{'},F^{'}$ with $Q^{''},R^{''},F^{''}$ respectively and repeating the above analysis $2k$ times, there are numbers $j_1,j_2$ with $2k+1\leq j_1\leq j_2\leq 2^{2k+1}$ such that $\sqcap_M(A_{j_1}^{'},B_{j_2}^{'})\geq k+1$ or $\sqcap_{M^*}(A_{j_1}^{'},B_{j_2}^{'})\geq k+1$, a contradiction to the fact that $\lambda(A_{j_1}^{'})=k$. So the lemma holds.
\end{proof}

To prove Theorem \ref{main-thm} we still need the following lemma.

\begin{lem}\label{S_1}(\cite{GGW07}, Lemma 4.7.)
Let $M$ be a matroid and $S,T$ be disjoint subsets of $E(M)$. There exists sets $S_1\subseteq S, T_1\subseteq T$ such that $|S_1|=|T_1|=\kappa(S_1,T_1)$.
\end{lem}

For convenience we restate Theorem \ref{main-thm} here.

\begin{thm}
There is a function $c: \Bbb{N}^2\rightarrow \Bbb{N}$ with the following property. Let $M$ be a matroid, and $Q,R,S,T,F\subseteq E(M)$ sets of elements such that $Q\cap R=S\cap T=\emptyset$ and $F=E(M)-(Q\cup R\cup S\cup T)$. Let $k:=\kappa_M(Q,R)$ and $\ell:=\kappa(S,T)$.  If $|F|\geq c(k,\ell)$, then there is an element $e\in F$ such that one of the following holds:
\begin{itemize}
    \item[(i)] $\kappa_{M\del e}(Q,R)=k$ and $\kappa_{M\del e}(S,T)=\ell$;
    \item[(ii)] $\kappa_{M/ e}(Q,R)=k$ and $\kappa_{M/ e}(S,T)=\ell$.
\end{itemize}
\end{thm}

\begin{proof}
We prove that the result holds for $c(k,\ell):=(2\ell+1)2^{2k+1}$. By Lemma \ref{S_1} there are sets $S_1\subseteq S, T_1\subseteq T$ such that $|S_1|=|T_1|=\kappa_M(S_1,T_1)$. Then Lemma \ref{step-1} implies that there is an element $e_1\in E(M)-(Q\cup R\cup S_1\cup T_1)$ such that for some $M_1\in\{M\del e_1, M/e_1\}$ we have  $\kappa_{M_1}(Q,R)=k$ and $\kappa_{M_1}(S_1,T_1)=\ell$. Since $\kappa_{M_1}(S_1,T_1)=\ell$ implies $\kappa_{M_1}(S,T)=\ell$, when $e_1\in F$ the lemma holds. So we may assume that $e_1\notin F$. That is, $e_1\in (S\cup T)-(S_1\cup T_1)$.  Since $F\subseteq E(M_1)-(Q\cup R\cup S_1\cup T_1)$, using Lemma \ref{step-1} again there is an element $e_2\in E(M_1)-(Q\cup R\cup S_1\cup T_1)$ such that  for some $M_2\in\{M_1\del e_2, M_1/e_2\}$ we have  $\kappa_{M_2}(Q,R)=k$ and $\kappa_{M_2}(S_1,T_1)=\ell$. Without loss of generality we may assume that $M_2=M_1\del e_2$. Then $\kappa_{M\del e_2}(Q,R)=k$ and $\kappa_{M\del e_2}(S_1,T_1)=\ell$ as $\kappa_{M}(Q,R)=k$ and $\kappa_{M}(S_1,T_1)=\ell$. Thus, when $e_2\in F$, the lemma holds. So we may assume that $e_2\notin F$. Since $(S\cup T)-(S_1\cup T_1)$ is finite, repeating the above analysis several times we can always find a minor with an element $e$ such that (i) or (ii) holds. The theorem follows from this observation and the fact that the
connectivity function is monotone under minors.
\end{proof}

\section{Acknowledgments}

The authors thank Tony Huynh and Stefan H. M. van Zwam for reading the paper and giving some helpful comments.

\end{document}